\documentclass[12pt]{amsart}
\usepackage{amsmath,amsfonts,amsthm,amscd, dcpic,graphicx}
\usepackage{tikz}
\usepackage{color}
\usetikzlibrary{arrows}
\theoremstyle{plain}
\newtheorem{thm}{Theorem}[section]
\newtheorem{prop}[thm]{Proposition}
\newtheorem{cor}[thm]{Corollary}
\newtheorem{lem}[thm]{Lemma}

\newtheorem{classification-thm}[thm]{Classification Theorem}
\newenvironment{proofmain}{\textbf{Proof of Theorem \ref{mainthmofourobstructiontheory}}:}{}

\theoremstyle{definition}

\newtheorem{remark}[thm]{Remark}
\newtheorem{defin}[thm]{Definition}

\begin{document}

\title{\bf Obstructions for constructing equivariant fibrations}
\author{Asl\i \ G\"{u}\c{c}l\"{u}kan \.Ilhan}

\address{Asli G\"u\c cl\" ukan \.Ilhan\\
Department of Mathematics \\
Bilkent University \\
Bilkent, Ankara, Turkey. }

\email{guclukan@fen.bilkent.edu.tr} 
\thanks{The author is supported by T\"UB\.ITAK-TBAG/110T712.}

\begin{abstract}
Let $G$ be a finite group and $\mathcal{H}$ be a family of
subgroups of $G$ which is closed under conjugation and taking
subgroups. Let $B$ be a $G$-CW-complex whose isotropy subgroups
are in $\mathcal{H}$ and let $\mathcal{F}= \{F_H\}_{H \in
\mathcal{H}}$ be a compatible family of $H$-spaces (see definition
\ref{def:compatible}). A $G$-fibration over $B$ with the fiber
type $\mathcal{F}= \{F_H\}_{H \in \mathcal{H}}$ is a
$G$-equivariant fibration $p:E \rightarrow B$ where $p^{-1}(b)$ is
$G_b$-homotopy equivalent to $F_{G_b}$ for each $b \in B$. In this
paper, we develop an obstruction theory for constructing
$G$-fibrations with the fiber type $\mathcal{F} $ over a given
$G$-CW-complex $B$. Constructing $G$-fibrations with a prescribed
fiber type $\mathcal{F}$ is an important step in the construction
of free $G$-actions on finite CW-complexes which are homotopy
equivalent to a product of spheres.
\end{abstract}
\maketitle
\section{Introduction} \label{sect:introductions}
In 1925, Hopf stated a problem which was later called the
topological spherical space form problem: Classify all finite
groups which can act freely on a sphere $\mathbb{S}^n$, $n
>1$. One variant of this problem was solved by Swan
\cite{Swan}. He proved that a finite group acts freely on a finite
complex homotopy equivalent to a sphere if and only if it has
periodic cohomology. By using Swan's construction and surgery
theory, the topological spherical space form problem has been
completely solved by Madsen-Thomas-Wall \cite{Madsen}. It turns
out that a finite group $G$ acts freely on a sphere if and only if
$G$ has periodic cohomology and any element of order $2$ in $G$ is
central (see \cite[Theorem 0.5]{Madsen}).

One of the generalizations of this problem is to classify all
finite groups which can act freely on a finite CW-complex homotopy
equivalent to a product of $k$-spheres $\mathbb{S}^{n_1} \times
\cdots \times \mathbb{S}^{n_k}$ for some $n_1, \dots, n_k$.
Recently, Adem and Smith \cite{Main} gave a homotopy theoretic
characterization of cohomological periodicity and as a corollary
they obtained a tool to construct free group actions on
CW-complexes homotopy equivalent to a product of spheres. More
precisely, they have shown that a connected CW-complex $X$ has
periodic cohomology if and only if there is a spherical fibration
over $X$ with a total space $E$ that has a homotopy type of a
finite dimensional CW-complex. As a consequence they proved that
if $G$ is a finite group and $X$ is a finite dimensional
$G$-CW-complex whose isotropy subgroups all have periodic
cohomology, then there is a finite dimensional CW-complex $Y$ with
a free $G$-action such that $Y \simeq \mathbb{S}^n \times X$. As
remarked in \cite{Main}, the second result can also be obtained
using the techniques given by Connolly and Prassidis in
\cite{Conolly-Prassidis}. More recently, Klaus \cite{Klaus} proved
that every $p$-group of rank $3$ acts freely on a finite
CW-complex homotopy equivalent to a product of three spheres using
similar techniques.

The method used by Connolly and Prassidis \cite{Conolly-Prassidis}
is to construct a spherical fibration inductively over the skeleta
by dealing with cells in each dimension separately. This is a
standard strategy in obstruction theory. Note that if there is an
orientable $G$-spherical fibration over the $n$-skeleton of a
CW-complex, then its restriction to the boundary of each
$(n+1)$-cell $\sigma$ will be an orientable $G_{\sigma}$-fibration
with the fiber $F_{G_{\sigma}}$ where $G_{\sigma}$ is the isotropy
subgroup of $\sigma$. Associated to this $G_{\sigma}$-fibration
over $\partial \sigma$, there is a classifying map from $\partial
\sigma$ to the space $B \mathrm{Aut_{G_{\sigma}}}F_{G_{\sigma}}$
where $\mathrm{Aut}_{G_{\sigma}}F_{G_{\sigma}}$ is the topological
monoid of self $G_{\sigma}$-homotopy equivalences of
$F_{G_{\sigma}}$. Combining the attaching map of $\sigma$ with the
classifying map gives us an element in the $n$-th homotopy group
of $B \mathrm{Aut_{G_{\sigma}}}F_{G_{\sigma}}$. Therefore we
obtain a cellular cochain which assigns a homotopy class in
$\pi_n(B \mathrm{Aut_{G_{\sigma}}}F_{G_{\sigma}})$ to each
$(n+1)$-cell. This cochain vanishes if and only if the
$G$-fibration over the $n$-skeleton extends to a $G$-fibration
over the $(n+1)$-skeleton. In the situation where
Connolly-Prassidis consider, this cochain can be killed by taking
fiber joins. Using this idea, \"{U}nl\"{u} \cite{Unlu-thesis}
gives a concrete cell-by-cell construction of $G$-spherical
fibrations in his thesis.

In obstruction theory, one often has obstructions as cohomology
classes which tells when a construction can be performed on the
next skeleton after some modifications. In other words, the
cohomological obstruction class vanishes if and only if the
restriction of the construction to the $(n-1)$-skeleton extends
over the $(n+1)$-skeleton. Having a cohomological obstruction is
better than having a cochain class as an obstruction since a
cohomology class is more likely to be zero. Note that if $p:E
\rightarrow B$ is a $G$-fibration and $b\in B^H$ then the fiber
$p^{-1}(b)$ is an $H$-space. When $B^H$ is connected for $H \leq
G$, there is an $H$-space $F_H$ such that for every $b \in B^H$,
the fiber $p^{-1}(b)$ is $H$-homotopy equivalent to $F_H$.
Moreover, if $B^H$ is connected for every $H \leq G$, the family
of $H$-spaces $F_H$ forms a compatible family (see
\ref{def:compatible} for a definition). In this case, the
$G$-fibration $p: E \rightarrow B$ is said to have the fiber type
$\{F_{H}\}$. In this paper, we notice that the cohomological
obstructions for constructing $G$-fibrations with a given fiber
type live in Bredon cohomology of $B$ with coefficients in
$\underline{\pi}_{n, \mathcal{F}}$ (see page \pageref{local coef.}
for the definition) and we prove the following theorem.
\begin{thm}\label{mainthmofourobstructiontheory} Let $G$ be a
finite group and
 $\mathcal{H}$ be a family of subgroups of $G$ which is closed
under conjugation and taking subgroups. Let $B$ be a
$G$-CW-complex whose isotropy subgroups are in $\mathcal{H}$ such
that $B^H$ is simply connected for every $H \in \mathrm{Iso}(B)$.
Let $\mathcal{F}= \{F_H\}_{H \in \mathcal{H}}$ be a compatible
family of finite $H$-CW-complexes and $p:E_n \rightarrow B^n$ be a
$G$-fibration over the $n$-skeleton of $B$ with the fiber type
$\{F_H\}_{H \in \mathcal{H}}$ where $n \geq 2$.
\begin{enumerate}
    \item There is a cocycle $\alpha_{p} \in C_{\mathcal{H}}^{n+1}(B;
\underline{\pi}_{n, \mathcal{F}})$ which vanishes if and only if
$p$ extends to a $G$-fibration over $B^{n+1}$ with a total space
$G$-homotopy equivalent to a $G$-CW-complex.
    \item The cohomology class $[\alpha_{p}] \in
H^{n+1}_{G, \mathcal{H}}(B; \underline{\pi}_{n, \mathcal{F}})$
vanishes if and only if the $G$-fibration $p|_{B^{n-1}}:
p^{-1}(B^{n-1}) \rightarrow B^{n-1}$ extends to a $G$-fibration
over $B^{n+1}$ with a total space $G$-homotopy equivalent to a
$G$-CW-complex.
\end{enumerate}

Moreover if $B$ is a finite $G$-CW-complex then the total space of
the obtained fibration has the $G$-homotopy type of a finite
$G$-CW-complex whenever $E_n$ has the $G$-homotopy type of a
finite $G$-CW-complex.
\end{thm}

To prove this theorem we first define an obstruction cochain in
the chain complex of Bredon cohomology and show that it is a
cocycle. We call this cocycle an obstruction cocycle. Then we show
that the difference of obstruction coycles of any two extensions
of the $G$-fibration $p|_{B^{n-1}}$ is the coboundary of a cochain
called the difference cochain. If there is an extension of
$p|_{B^{n-1}}$ to a $G$-fibration over $B^{n+1}$, then the
obstruction cocycle of the restriction of this extension to $B^n$
vanishes. This means that the obstruction cocycle of $p$ is a
coboundary and represents a cohomology class which vanishes. This
proves the ``if" direction of the above theorem.

For the ``only if" direction it suffices to show that for every
cochain $d$ there is a $G$-fibration $q$ over $B^n$ with
$q|_{B^{n-1}}= p|_{B^{n-1}}$ such that $d$ is the difference
cochain
 of the extensions $p$ and $q$ of $p|_{B^{n-1}}$.
Here the most technical part is the construction of
 a $G$-fibration $q$ with these properties. That is because it is not clear how to glue $G$-fibration $p|_{B^{n-1}}$
 to $G$-fibrations over the $n$-cells corresponding to the cochain $d$. For quasifibrations it
 suffices to take the adjunction of the total
 spaces to glue two quasifibrations over different base. However, in order to obtain a fibration
 one needs to put some $G$-tubes between
total spaces of these $G$-fibrations
 to create enough space to deal with $G$-homotopies. We use a
 generalization of a result due to Tulley \cite{Tulley} to produce
 a $G$-fibration $q$ with the required properties.

The paper is organized as follows: Section \ref{sect:equivariant
fibrations} contains definitions and preliminary results on
equivariant fibrations. In Section \ref{sect:Tulley'stheorem}, we
give a method to glue $G$-fibrations over different base spaces by
generalizing a construction due to Tulley \cite{Tulley}. Finally,
we prove Theorem \ref{mainthmofourobstructiontheory} in Section
\ref{sect:constructing G-fibrations}.
\section{Equivariant fibrations}
\label{sect:equivariant fibrations} In this section, we give the
basic definitions of the equivariant fibration theory.
We refer the reader to \cite{Luck} and \cite{Waner} for more
details.

\begin{defin} Let $G$ be a finite group. A $G$-map $p: E \rightarrow B$
is called a $G$-fibration if it has the $G$-homotopy lifting
property with respect to every $G$-space $X$, that is, given a
commutative diagram of $G$-maps
\begin{center}
\begin{tikzpicture}[node distance=2cm, auto]
 \node (P) {$X \times \{0\}$};
  \node (E) [right of=P] {$E$};
  \node (A) [below of=P] {$X\times I$};
  \node (B) [below of=E] {$B$};
  \draw[->] (P) to node {$h$} (E);
  \draw[right hook->] (P) to node [swap] {} (A);
  \draw[->] (A) to node [swap] {$H$} (B);
  \draw[->] (E) to node {$p$} (B);
\draw[->, dashed] (A) to node {$\widetilde{H}$} (E);
\end{tikzpicture}
\end{center}
there exists a $G$-map $\widetilde{H}: X \times I \rightarrow E$
such that $p\widetilde{H}= H$ and $\widetilde{H}|_{X
\times\{0\}}=h$.
\end{defin}

Equivalently, a $G$-map $p:E \rightarrow B$ is a $G$-fibration if
there is a $G$-map $$\lambda: \Omega_p=\{(e,\omega) \in E \times
B^{I}| \ p(e)= \omega(0)\} \rightarrow E^I$$ such that $\lambda(e,
\omega)(0)=e$ and $p\lambda(e,\omega)=\omega$. The $G$-map
$\lambda$ is called a {\it $G$-lifting function}. By using a
similar definition, Dold \cite{Dold} proved that being a fibration
is a local property. The same proof applies to the equivariant
case.

\begin{defin}A covering $\mathcal{U}$ of $G$-invariant open sets of $B$ is
called {\it numerable $G$-covering} if $\mathcal{U}$ is locally
finite and there is a $G$-map $f_U:B \rightarrow I$ such that
$U=f_U^{-1}(0,1]$ for every $U\in \mathcal{U}$.
\end{defin}


\begin{thm} \label{uniformization theorem} A $G$-map $p:E \rightarrow B$
is a $G$-fibration if there is a numerable $G$-covering
$\mathcal{U}$ of $B$ such that $p|_U: p^{-1}(U) \rightarrow U$ is
a $G$-fibration for each $U \in \mathcal{U}$.
\end{thm}
The notion of an equivalence between $G$-fibrations is defined
naturally as follows: Let $p_i:E_i \rightarrow B$ be a
$G$-fibration for $i=1,2$. A fiber preserving $G$-map $f:E_1
\rightarrow E_2$ is called a {\it $G$-fiber homotopy equivalence}
if there is a fiber preserving $G$-map $g:E_2 \rightarrow E_1$
such that the compositions $fg$ and $gf$ are $G$-homotopy
equivalent to identity maps through $G$-homotopies which are fiber
preserving at each time $t\in I$. In this case, we write $p_1
\simeq_G p_2$. As in the non-equivariant case, a fiber preserving
$G$-homotopy equivalence between $G$-fibrations is a $G$-fiber
homotopy equivalence (see \cite[pg. 50]{May} for the proof of the
non-equivariant case).

In \cite{Stasheff}, Stasheff proved a classification theorem for
non-equivariant fibrations up to fiber homotopy equivalences. When
a $G$-fibration is over a path-connected space with trivial
$G$-action, the fiber at each point in the base has a natural
$G$-space structure and all fibers are $G$-homotopy equivalent
with respect to this structure. In this case, the theory of
$G$-fibrations is essentially the same as the non-equivariant one
and we have the following classification theorem.

\begin{thm}\label{classification theorem}Let $\mathrm{Aut}_G(F_G)$
be the monoid of $G$-equivariant self homotopy equivalences of a
finite $G$-CW-complex $F_G$. If $B$ is a CW-complex with trivial
$G$-action then there is a one-to-one correspondence between the
set of $G$-fiber homotopy equivalence classes of $G$-fibrations
$p:X \rightarrow B$ with fibers having the $G$-homotopy type of
$F_G$ and the set of homotopy classes of maps $B \rightarrow B
\mathrm{Aut}_G(F_G)$. The equivalence is obtained by taking the
pullback fibration from the universal fibration over
$B\mathrm{Aut}_G(F_G)$.
\end{thm}

This classification theorem can be proved by using the same
techniques and ideas from \cite{Stasheff}. Also, Waner constructs
a classifying space for a more general set of equivariant
fibrations in \cite{WanerClass} and the above theorem can be
obtained as a special case of his result.

The monoid $\mathrm{Aut}_G(F_G)$ is not connected in general.
However, its connected components are homotopy equivalent via the
maps $(\mathrm{Aut}_G(F_G),f) \rightarrow (\mathrm{Aut}_G(F_G),g)$
with $\phi \rightarrow g\phi f^{-1}$ where $f^{-1}$ is the
homotopy inverse of $f$. Furthermore, when the map $\pi_1(B)
\rightarrow [F_G, F_G]_G$ is trivial, $B\mathrm{Aut}_G^I(F_G)$
classifies $G$-fibrations $p:E \rightarrow B$ with trivial
$G$-actions on the base where $\mathrm{Aut}^I_G(F_G)$ is the
connected component of identity in $\mathrm{Aut}_G(F_G)$ (see
\cite{Main} for non-equivariant case).

For $G$-fibrations whose $G$-action on the base is not trivial, we
need to consider the collection of equivariant spaces. Note that
if $p$ is a $G$-fibration over $B$ and $b \in B^K$, then the space
$p^{-1}(b)$ is closed under $K$-action and hence a $K$-space.
Moreover, when $B^K$ is connected, the spaces $p^{-1}(b)$ and
$p^{-1}(b')$ are $K$-homotopy equivalent for every $b, b' \in
B^K$. On the other hand, when $H^a \leq K$, we have $h(ab)=
a(a^{-1}ha)b=ab$ for every $h \in H$ and $b\in B^K$, hence $ab \in
B^H$. Clearly, the $H$-space $p^{-1}(b)$, where the $H$-action on
$p^{-1}(b)$ is given by conjugation and the $H$-space $p^{-1}(ab)$
are $H$-homeomorphic. Therefore, when $B^H$ is connected for every
isotropy subgroup $H$ of $B$, the spaces $p^{-1}(b)$ and
$p^{-1}(ab)$ are $H$-homotopy equivalent for every isotropy
subgroups $H,K$ with $H^a \leq K$.

\begin{defin}\label{def:compatible}Let $\mathcal{H}$ be a family of subgroups of $G$ which is
closed under conjugation and taking subgroups. A family
$\mathcal{F}=\{F_H\}_{H\in \mathcal{H}}$ of $H$-spaces is said to
be a {\it compatible family} if $F_H$ is $H$-homotopy equivalent
to $F_K$ for every $H,K \in \mathcal{H}$ with $H^a \leq K$ for
some $a\in G$ where the $H$-action on $F_K$ is given by $h\cdot
y=a^{-1}hay$.
\end{defin}

Let $\mathcal{F}=\{F_{H}\}_{H \in \mathcal{H}}$ be a compatible
family where $\mathcal{H}$ contains the isotropy subgroups of $B$.
We say $p:E \rightarrow B$ is a {\it $G$-fibration with the fiber
type} $\mathcal{F}=\{F_H\}_{H \in \mathcal{H}}$ if $F_H \simeq_H
p^{-1}(b)$ for every $b \in B^H$ and for every isotropy subgroup
$H$ of $B$. When $B^H$ is connected for every $H \in \mathcal{H}$,
every $G$-fibration over $B$ is a $G$-fibration with the fiber
type $\mathcal{F}$. However, a $G$-fibration $p:E \rightarrow B$
does not necessarily have a fiber type.

\section{Tulley's theorem for $G$-fibrations}
\label{sect:Tulley'stheorem}

The aim of this section is to prove an equivariant version of a
theorem due to Tulley (see \cite[Theorem 11]{Tulley}). The proof
uses the same ideas and methods from \cite{Tulley} and
\cite{Tulley2}.

\begin{thm}\label{thm:equiv. version of Tulley} Let $p_1:E_1 \rightarrow B$ and $p_2: E_2 \rightarrow
B$ be $G$-fiber homotopy equivalent $G$-fibrations. Then there is
a $G$-fibration $q$ over $B \times I$ such that $q|_{B \times
\{0\}}=p_1$ and $q|_{B \times \{1\}}=p_2$.
\end{thm}

We call the $G$-fibration $q:Z \rightarrow B \times I$ in Theorem
\ref{thm:equiv. version of Tulley} a {\it $G$-tube} between $p_1$
and $p_2$. Let $f:E_1 \rightarrow E_2$ be a fiber preserving
$G$-map between $G$-fibrations $p_1$ and $p_2$ over $B$. Recall
that the mapping cylinder $M_f$ of $f$ is the adjunction space
$E_1 \times I \cup_{f} E_2$ where $f(e)=(e,0)$ for each $e\in
E_1$. We define the $G$-map $p_f:M_f \rightarrow B$ over $B$ by
$p_f(x,s)=p_1(x)$ and $p_f(y)=p_2(y)$ for any $x \in E_1$, $y \in
E_2$, and $s\in I$.

\begin{lem}\label{lem:mapping cylinder} Let $f:E_1 \rightarrow E_2$ be a fiber preserving $G$-map between
$G$-fibrations $p_1$ and $p_2$ over $B$. Then the induced $G$-map
$p_f:M_f \rightarrow B$ is a $G$-fibration.
\end{lem}

\begin{proof} The proof is similar to the proof of Theorem 1 in
\cite{Tulley2}. Let $\lambda_i: \Omega_{p_i} \rightarrow {E_i}^I$
be a $G$-lifting function for $p_i$, $i=1,2$. Since $\Omega_{p_f}=
\Omega_{p_1} \times I \cup_{\widetilde{f}} \Omega_{p_2}$ where
$\widetilde{f}(e,\omega)=(f(e), \omega)$, to show that $p_f$ is a
$G$-fibration it suffices to construct a $G$-map $\lambda:
\Omega_{p_1} \times I \rightarrow (M_f)^I$ such that
$\lambda|_{\Omega_{p_1} \times \{0\}}= \lambda_{2}\circ
\widetilde{f}$, $p_f \lambda((e,\omega),s)=\omega$, and
$\lambda((e, \omega),s)(0)=(e,s)$.

Define $\lambda: \Omega_{p_1} \times I \rightarrow M_f^I$ by
\begin{eqnarray*} \lambda(e,\omega, s)(t)= \left\{%
\begin{array}{ll}
    (\lambda_1(e,\omega)(t), s-t), & \hbox{$t \leq s$,} \\
    \lambda_2(z, \omega^s)(t-s), & \hbox{$s \leq t$,} \\
\end{array}%
\right.
\end{eqnarray*}
where $z=f\circ \lambda_1(e, \omega)(s)$ and $\omega^{s}$ is given
by $\omega^{s}(t)= \omega(s+t)$ when $s+t \leq 1$ and
$\omega^s(t)=1$, otherwise. Clearly, $\lambda$ is a continuous
$G$-map which satisfies the relations $\lambda|_{\Omega_{p_1}
\times \{0\}}= \lambda_{2}\circ \widetilde{f}$, $p_f
\lambda((e,\omega),s)=\omega$, and $\lambda((e,
\omega),s)(0)=(e,s)$.
\end{proof}

In order to prove Theorem \ref{thm:equiv. version of Tulley}, it
suffices to construct $G$-fibrations $q_1:Z_1 \rightarrow B \times
I$ and $q_2: Z_2 \rightarrow B \times I$ with $q_1|_{B \times
\{0\}}=p_1$, $q_1|_{B \times \{1\}}= q_2|_{B \times \{0\}}=p_f$
and $q_{2}|_{B \times \{1\}}=p_2$ where $f$ is the fiber
preserving $G$-map between $p_1$ and $p_2$. That is because once
we have such $G$-fibrations, we can obtain a $G$-tube between
$p_1$ and $p_2$ by gluing $q_1$ and $q_2$ as follows. Let $Z= Z_1
\cup_{i_1} M_f \times I \cup_{i_2} Z_2$ where $i_1(x)=(x,0)$ and
$i_2(x)=(x,1)$ for every $x \in M_f$. Define the $G$-map $q:Z
\rightarrow B\times I$ by
\begin{eqnarray*}q(z)=\left\{%
\begin{array}{ll}
   (\pi_1(q_{1}(z)), \ \frac{1}{3}\pi_2(q_{1}(z))), & \hbox{$z \in Z_1$;} \\
  (\pi_1(q_{2}(z)), \ \frac{2}{3}+\frac{1}{3}\pi_2(q_{2}(z))), & \hbox{$z \in Z_2$;} \\
    (p_f(x), \frac{1}{3}(1+t)), & \hbox{$z=(x,t) \in M_f \times I$} \\
\end{array}%
\right.
\end{eqnarray*} where $\pi_i$ is the projection map to the $i$-th coordinate.
Since we can  extend the lifting functions for $q_1$ and $q_2$ to
lifting functions for $q|_{B\times [0, \frac{6}{9})}$ and $q|_{B
\times (\frac{4}{9}, 1]}$, respectively, $G$-maps $q|_{B \times
[0, \frac{6}{9})}$ and  $q_{B \times (\frac{4}{9},1]}$ are
$G$-fibrations. Therefore $q:Z \rightarrow B \times I$ is a
$G$-fibration by Theorem \ref{uniformization theorem}.

\begin{prop} Let $Z_2=\{(e,s,t) \in E_1 \times I \times I | \
s+t \leq 1 \} \cup_{\widetilde{f}} E_2 \times I$ in $M_f \times I$
where $\widetilde{f}: E_1 \times \{0\} \times I$ is defined by
$\widetilde{f}(e,0,t)=(f(e),t)$. Then $q_{2}= (p_f \times
\mathrm{id})|_{Z_2}: Z_2 \rightarrow B \times I$ is a
$G$-fibration with $q_2|_{B \times \{0\}}=p_f$ and $q_{2}|_{B
\times \{1\}}=p_2$.
\end{prop}

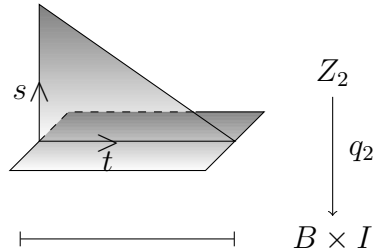
\begin{center}\begin{tikzpicture}[scale=1.3]
\shade (0,2) -- (2,2)--(0,3.4)-- cycle;  \shade (-0.3,1.7) --
(1.7,1.7) --(2.3,2.3)--(0.3,2.3)-- cycle; \draw (0,2) -- (2,2) --
(0,3.4) -- cycle; \draw (0,2)--(-0.3,1.7) -- (1.7,1.7) -- (2,2);
\draw[-,dashed] (0,2)-- (0.3,2.3) -- (1.6,2.3);
\draw(1.6,2.3)--(2.3,2.3)--(2,2); 
\path (3,2.7) node (Y) {$Z_2$}; \draw[|-|] (-0.2,1) to (2,1);
\path (3,1) node (X) {$B \times I$}; \path (0,2.5) node (A)
{$\wedge$}; \path (-0.2, 2.5) node (B) {$s$}; \path (0.7, 1.98)
node (D) {$>$}; \path (0.7, 1.8) node(E) {$t$};
  \draw[->] (Y) to node {} (X);
 \path (3.3,1.9) node (Z) {$q_2$};
\end{tikzpicture}
\end{center}

\begin{proof} Let $r: M_f \times I \rightarrow Z_2$ be defined by $r|_{E_2 \times I}=
\mathrm{id}_{E_2 \times I}$ and
\begin{eqnarray*} r|_{E_1 \times I \times I}(x,s,t)=\left\{%
\begin{array}{ll}
(x,s,t) & \hbox{$s+t \leq 1$;} \\
    (x,t,t), & \hbox{otherwise.} \\
\end{array}%
\right.
\end{eqnarray*}
Then $r$ is a fiber preserving $G$-retraction. Let $H: X \times I
\rightarrow B$ and $h: X \rightarrow Z_2$ be given $G$-maps with
$H|_{X \times \{0\}}=p_2 \circ h$. Since $p_f \times \mathrm{id}$
is a $G$-fibration, there is a $G$-map $\bar{H}: X \times I
\rightarrow M_f \times I$ which makes the following diagram
commute:
\begin{center}
\begin{tikzpicture}[node distance=2.3cm, auto]
  \node (P) {$X \times \{0\}$};
  \node (D) [right of=P] {$Z_2$};
  \node (E) [right of=D] {$M_f \times I$};
  \node (A) [below of=P] {$X \times I$};
  \node (B) [below of=D] {$B \times I$};
  \node (C) [below of=E] {$B \times I$};
  \path (3.1,-1.05) node (X) {$\bar{H}$};
  \path (2.1,-1) node (Y) {$p_2$};
  \draw[->] (P) to node {$h$} (D);
  \draw[right hook->] (P) to node [swap] {} (A);
  \draw[->] (A) to node [swap] {$H$} (B);
  \draw[->] (D) to node [swap] {} (B);
  \draw[->] (E) to node {$ p_f \times \mathrm{id}$} (C);
  \draw[->] (A) to node {} (E);
  \draw[->, dashed] (A) to node {$\widetilde{H}$} (D);
  \draw[double, -] (B) to node [swap] {} (C);
  \draw[->] (E)[bend right] to node [swap] {$r$} (D);
  \draw[right hook->] (D) to node {$i$} (E);
\end{tikzpicture}
\end{center} Then the $G$-map $\widetilde{H}: X\times I \rightarrow Z_2$
defined by $\widetilde{H}=r\bar{H}$ satisfies $p_2
\widetilde{H}=H$ and $\widetilde{H}|_{X \times \{0\}}= h$.
\end{proof}

\begin{defin} Let $p_i: E_i \rightarrow B$
be a $G$-fibration for $i=1,2$ with $E_1 \subseteq E_2$ and
$p_2|_{E_1}=p_1$. Then $p_1$ is said to be a $G$-deformation
retract of $p_2$ if $E_1$ is a $G$-deformation retract of $E_2$
via fiber preserving $G$-retraction, that is, if there is a
$G$-map $H:E_2 \times I \rightarrow E_2$ such that $H_0=
\mathrm{id}_{E_2}$, $H(e,1) \in E_1$ and $p_2H(e,t)=p_2(e)$ for
every $e \in E_2$ . If $H$ also satisfies the relation $H(e,t)=e$
for every $e \in E_1$, we say $p_1$ is a strong $G$-deformation
retract of $p_2$.
\end{defin}

To show that there is a $G$-tube $q_1: Z_1 \rightarrow B \times I$
between $p_1$ and $p_f$, we need the following proposition. The
non-equivariant version of this proposition is proved in
\cite{Tulley} and used in a recent paper by Steimle
\cite{Steimle}.

\begin{prop}\label{sdr-sfhe} If $p_1$ is a strong $G$-deformation
retract of $p_2$ then the $G$-map $q=(p_{2} \times
\mathrm{id})|_{Z}:Z \rightarrow B\times I$ where $Z=\{(e,t) \in
E_2 \times I| \ e \in E_2 \ \mathrm{if} \ t>0, \ e \in E_1 \
\mathrm{if} \ t=0\}$ is a $G$-fibration.
\end{prop}

\begin{center}\begin{tikzpicture}[scale=1.5]
\shade (0,2)--(2,2)--(2,3.4)--(0,3.4)--cycle; \draw
(0,3.4)--(2,3.4)--(2,2)--(0,2); \path (3,3) node (Y) {$Z$}; \draw
(0, 2.4)--(0,2.9); \draw[-,dashed] (0,2)--(0,2.4); \draw[-,dashed]
(0,2.9)--(0,3.4); \draw[|-|] (0,1.3) to (2,1.3); \path (3,1.3)
node (X) {$B \times I$}; \path(-0.2,2.7) node (Z) {$E_1$};
  \draw[->] (Y) to node {} (X);
 \path (3.3,2) node (Z) {$q$};
\end{tikzpicture}
\end{center}

\begin{proof} We are using the same approach that is used in the
proof of Theorem 3.1 in \cite{Langston}. Let $H: E_2 \times I
\rightarrow E_2$ be a strong $G$-deformation retraction of the
$G$-fibration $p_2$ onto $p_1$. Let $\lambda: \Omega_{p_2}
\rightarrow E_2^I$ be a $G$-lifting function for $p_2$. Define a
$G$-map $\lambda': \Omega_{q} \rightarrow Z^I$ by $\pi_2 \big(
\lambda'((e,\omega_2(0)), (\omega_1, \omega_2))(t)\big )=
\omega_2(t)$ and
$$\pi_1 \big ( \lambda'((e,\omega_2(0)), (\omega_1, \omega_2))(t)\big ) =\left\{%
\begin{array}{ll}
    H(\lambda(e,\omega_1)(t), \frac{t}{\omega_2(t)}), & \hbox{$\omega_2(t)>0, \ \omega_2(t) \geq t$;} \\
    e, & \hbox{$t=\omega_2(t)=0$;} \\
    H(\lambda(e,\omega_1)(t), 1), & \hbox{$t>0, \ t \geq \omega_2(t)$.} \\
\end{array}%
\right.    $$ Clearly, $q \lambda'((e,\omega_2(0)), (\omega_1,
\omega_2))= (\omega_1,\omega_2)$ and $\lambda' ((e,\omega_2(0)),
(\omega_1, \omega_2))(0)=e$. Therefore we only need to check the
continuity of $ \pi_1 \lambda'$ at $t=0$. For this it suffices to
show that the adjoint map $\widetilde{\pi_1\lambda'}: \Omega_q
\times I \rightarrow E_2$ is continuous at $t=0$.

Let $(e_{\alpha},\omega_{1, \alpha},\omega_{2,\alpha},
t_{\alpha})$ be a net converging to $(e,\omega_1,\omega_2,0)$. Let
$U$ be an open neighborhood of $e \in E_1$. Since $H: E_2 \times I
\rightarrow E_2$ is continuous, $V= H^{-1}(U)$ is open. Since
$(e,t) \in V$ for every $t \in I$, there are open neighborhoods
$A_t \ni e$ and $V_t \ni t$ such that $A_t \times V_t \subseteq
V$, for all $t \in I$. Since $I$ is compact, there exist $t_1,
\dots , t_n$ such that $I= \cup_{i=1}^n V_{t_i}$. Then $A=
\cap_{i=0}^n A_{t_i}$ is an open neighborhood of $e$ with the
property that $H(A \times I) \subseteq U$. Since $\lambda$ is
continuous, there is $\beta$ such that
$\widetilde{\lambda}(e_{\alpha}, \omega_{1,\alpha}, t_{\alpha})
\in A$ for every $\beta > \alpha$. Therefore
$\widetilde{\pi_1\lambda'}(e_{\alpha},\omega_{1,
\alpha},\omega_{2,\alpha}, t_{\alpha}) \in U$ for every $\alpha >
\beta$ as desired.
\end{proof}

It is proved in \cite[Corollary 2.4.2]{Piccinini} that when $f:
E_1 \rightarrow E_2$ is a homotopy equivalence then $E_1$ is a
strong deformation retract of $M_f$. The same proof applies to
$G$-fibrations (see \cite[Lemma 2.5.2]{Guclukan}). Therefore, by
Proposition \ref{sdr-sfhe}, there is a $G$-tube $q_1$ between
$p_1$ and $p_f$. Note that $p_2$ is also a strong $G$-deformation
retract of $p_f$ and one can also use Proposition \ref{sdr-sfhe}
to construct a $G$-tube between $p_2$ and $p_f$. This completes
the proof of Theorem \ref{thm:equiv. version of Tulley}. As an
immediate corollary, we have:

\begin{cor}\label{cor:genel} Let $B_1,B_2,$ and $B$ be topological spaces such that
$B \subseteq B_1 \cap B_2$. If $p_1:E_1 \rightarrow B_1$ and $p_2:E_2
\rightarrow B_2$ are $G$-fibrations with $p_1|_{B} \simeq
p_2|_{B}$ then there is a $G$-fibration over $B_1 \cup_{i_1} (B
\times I) \cup_{i_2} B_2$ extending $p_1$ and $p_2$ where $i_j: B
\times \{j\} \rightarrow B_j$ are inclusions.
\end{cor}

\begin{proof} By Theorem \ref{thm:equiv. version of Tulley}, there is a
$G$-tube $q:Z \rightarrow B \times I$ between $p_1|_{B}$ and
$p_2|_B$. Without loss of generality, we can consider $q$ over $B
\times [\frac{1}{3},\frac{2}{3}]$ with $q|_{B \times
\{\frac{1}{3}\}}= p_1$ and $q|_{B \times \{\frac{2}{3}\}}= p_2$.
Let $$\widetilde{Z}= E_1 \cup_{k_1} \Big( p_1^{-1}(B) \times [0,
\frac{1}{3}] \Big ) \cup_{m_1} Z \cup_{m_2} \Big( p_2^{-1}(B)
\times [\frac{2}{3},1] \Big) \cup_{k_2} E_2$$ where $k_1:
p_1^{-1}(B) \times \{0\} \rightarrow E_1$, $k_2: p_1^{-1}(B)
\times \{1\} \rightarrow E_2$ and $m_1: p_1^{-1}(B) \times
\{\frac{1}{3}\} \rightarrow Z$, $m_2: p_2^{-1}(B) \times
\{\frac{2}{3}\} \rightarrow Z$ are the inclusions. Define a
$G$-map $$\widetilde{q}:\widetilde{Z} \rightarrow B_1 \cup_{i_1}
\big( B \times I \big) \cup_{i_2} B_2$$ by
$\widetilde{q}|_{E_i}=p_i$, $\widetilde{q}|_{Z \times
[\frac{1}{3}, \frac{2}{3}]}(z,t)= q(z)$ and by
$\widetilde{q}(e,t)= p_j(e)$ for $t \in [\frac{2(j-1)}{3},
\frac{2j-1}{3}])$, $j=1,2$.
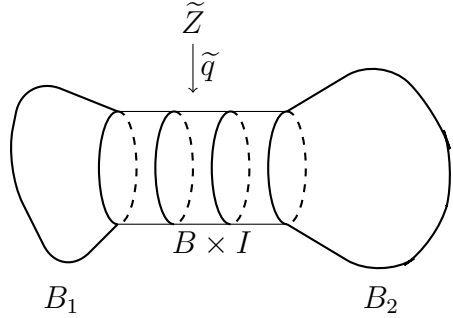
\begin{center}\begin{tikzpicture}[scale=0.5] \draw [rounded
corners=15pt, thick] (3,4.5)--(0.5,5.5)--(0,3)--(1.5,0)--(3,1.5);
\draw[-] (3,4.5)--(7.5,4.5); \draw[-] (3,1.5)--(7.5,1.5); \draw
[rounded corners=15pt, thick]
(7.5,4.5)--(10,6)--(11.5,4.5)--(12,3)--(11.5,1)--(10,0)--(7.5,1.5);
\draw[thick] (3,1.5) arc (270:90:0.5cm and 1.5cm);
\draw[dashed,thick] (3,4.5) arc (90:-90:0.5cm and 1.5cm);
\draw[thick] (7.5,1.5) arc (270:90:0.5cm and 1.5cm);
\draw[dashed,thick] (7.5,4.5) arc (90:-90:0.5cm and 1.5cm);
\draw[thick] (4.5,1.5) arc (270:90:0.5cm and 1.5cm);
\draw[dashed,thick] (4.5,4.5) arc (90:-90:0.5cm and 1.5cm);
\draw[thick] (6,1.5) arc (270:90:0.5cm and 1.5cm);
\draw[dashed,thick] (6,4.5) arc (90:-90:0.5cm and 1.5cm);
\path (1.5, -0.5) node (Z) {$B_1$}; \path (10, -0.5) node (Y)
{$B_2$}; \path (5.5,1) node (X) {$B\times I$}; \path (5,7) node
(A) {$\widetilde{Z}$}; \draw[->] (A) to (5,5); \path (5.4,5.7)
node (B) {$\widetilde{q}$};
\end{tikzpicture}
\end{center}
Clearly the restriction of $\widetilde{q}$ to the following
subsets are $G$-fibrations
$$\{B_1 \cup_{i_1} B \times [0, \frac{2}{9}], B \times [\frac{1}{9}, \frac{5}{9}],
B \times [\frac{4}{9},\frac{8}{9}], B \times [\frac{7}{9},1]
\cup_{i_2}E_2 \}$$ Therefore, by Theorem \ref{uniformization
theorem}, $\widetilde{q}$ is a $G$-fibration.
\end{proof}

\begin{thm} \label{mainthmfin} Let $p_i:E_i \rightarrow B$, $i=1,2$, be $G$-fiber homotopy equivalent
$G$-fibrations such that $E_1$ and $E_2$ have the $G$-homotopy
type of a $G$-CW-complex. Then there is a $G$-tube $q$ between
$p_1$ and $p_2$ whose total space is $G$-homotopy equivalent to a
$G$-CW-complex.
\end{thm}

\begin{proof} Let $f: E_1 \rightarrow E_2$ be a $G$-fiber homotopy equivalence. Recall that the total space of the
$G$-tube $q$ we constructed is $Z= Z_1 \cup_{i_1} M_f \times
[\frac{1}{3},\frac{2}{3}] \cup Z_2$ where
\begin{eqnarray*}
Z_1&=&\{(e,t) \in M_f \times [0, \frac{1}{3}]| \ e \in M_f \
\mathrm{if} \ 0< t \leq \frac{1}{3}, \ e \in E_1 \ \mathrm{if} \ t=0\}\\
Z_2&=&\{(x,s,t) \in E_1 \times I \times [\frac{2}{3},1] | \ s+3t
\leq 3 \} \cup_f E_2 \times [\frac{2}{3},1].
\end{eqnarray*}
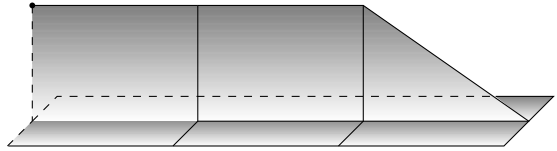
\begin{center}\begin{tikzpicture}[scale=1.1]
\shade (0,2)--(0,3.4)--(4,3.4)--(6,2)--cycle; \shade
(0,2)--(6,2)--(5.7,1.7)--(-0.3,1.7)--cycle; \shade
(5.6,2.3)--(6.3,2.3)--(6,2)--cycle; \draw[-] (2,2)--(6,2);
\draw[-] (5.6,2.3)--(6.3,2.3)--(6,2); \draw[-]
(1.7,1.7)--(5.7,1.7)--(6,2)--(4,3.4)--(2,3.4); \draw
(4,3.4)--(4,2)--(3.7,1.7); \draw[-] (2,3.4)--(2,2)--(1.7,1.7);
\draw[-] (-0.29,1.7)--(1.7,1.7); \draw[-,dashed]
(-0.3,1.7)--(0,2); \draw[-,dashed] (0,3.4)--(0,2); \draw[-,dashed]
(0,2)--(0.3,2.3)--(5.6,2.3); \draw(2,3.4)--(0,3.4); \fill (0,3.4)
circle (1pt); 
\end{tikzpicture}\end{center}
Clearly, $Z$ is a strong $G$-deformation retract of $M_f$. On the
other hand, $M_f$ is $G$-homotopic to $E_2$ and hence it has a
$G$-homotopy type of a $G$-CW-complex.
\end{proof}

\begin{cor}\label{maincor} A $G$-fibration $p:E \rightarrow \mathbb{S}^{n-1}$ over $(n-1)$-sphere
with trivial $G$-action on the base extends to a $G$-fibration
over a disk if and only if it is $G$-fiber homotopy equivalent to
a trivial $G$-fibration.
\end{cor}

\begin{proof} Since $\mathbb{D}^{n}$ is contractible, the ``only if"
part holds. For the ``if" part, let $p$ be $G$-fiber homotopy
equivalent to a trivial fibration. Consider $\mathbb{D}^n$ as the
cone of $\mathbb{S}^{n-1}$, that is, $\mathbb{D}^n=
\mathbb{S}^{n-1}\times [0,2]/ \sim $ where $(y,2)\sim \ast$. Let
$B_1= S^{n-1}\times[1,2]/ \sim$, and $B_2=B=\mathbb{S}^{n-1}\times
\{1\}$. Then $\varepsilon|_{B} \simeq p$ where
 $\varepsilon: B_1 \times F_G \rightarrow B_1$ where
$F_G=p^{-1}(x)$ for some $x \in \mathbb{S}^{n-1}$. By Corollary
\ref{cor:genel}, there is a $G$-fibration extending $p$.
\end{proof}

\begin{remark} In \cite{Conolly-Prassidis}, the
statement of Corollary \ref{maincor} appears on page 137 but in
there the total spaces were attached directly which results in a
$G$-quasifibration from which one gets a $G$-fibration by taking
the corresponding Hurewicz fibration.
\end{remark}

\section{Constructing $G$-fibrations}
\label{sect:constructing G-fibrations}

In this section, we introduce an obstruction theory for
constructing $G$-fibrations over $G$-CW-complexes and we prove
Theorem \ref{mainthmofourobstructiontheory}. An adequate
cohomology theory to develop such an obstruction theory is Bredon
cohomology. Let us first fix some notation for Bredon cohomology.
We refer the reader to \cite{BredonLN} and \cite{Luck} for more
detailed information about Bredon cohomology.

Let $G$ be a finite group and $\mathcal{H}$ be a family of
subgroups of $G$ which is closed under conjugation and taking
subgroups. The orbit category $\mathrm{Or}_{\mathcal{H}}(G)$
relative to the family $\mathcal{H}$ is defined as the category
whose objects are left cosets $G/H$ where $H \in \mathcal{H}$ and
whose morphisms are $G$-maps from $G/H$ to $G/K$. Recall that any
$a \in G$ with $H^a \leq K$ induces a $G$-map $\hat{a}:G/H
\rightarrow G/K$ defined by $\hat{a}(H)=aK$ and conversely, any
$G$-map from $G/H$ to $G/K$ is of this form.

Let $B$ be a $G$-CW-complex whose isotropy subgroups are all in
$\mathcal{H}$. A coefficient system for the Bredon cohomology is a
contravariant functor $M:\mathrm{Or}_{\mathcal{H}}(G) \rightarrow
\mathcal{A}b$ where $\mathcal{A}b$ is the category of abelian
groups. A morphism $T:M \rightarrow N$ between two coefficient
systems is a natural transformation of functors. Note that a
coefficient system is a
$\mathbb{Z}\mathrm{Or}_{\mathcal{H}}(G)$-module with the usual
definition of modules over a small category. Since the
$\mathbb{Z}\mathrm{Or}_{\mathcal{H}}(G)$-module category is an
abelian category, the usual notions for doing homological algebra
exist.

Given a local coefficient system $M:\mathrm{Or}_{\mathcal{H}}(G)
\rightarrow \mathcal{A}b$,  one defines the cochain complex
$C^{\ast}_{\mathcal{H}}(B; M)$ of $B$ with coefficients in $M$ as
follows: Let $C^n_{\mathcal{H}}(B; M)$ be the submodule of
$\oplus_{H\in \mathcal{H}}\mathrm{Hom}_{\mathbb{Z}}(C_n(B^H;
\mathbb{Z}), M(G/H))$ formed by elements $(f(H))_{H \in
\mathcal{H}}$ which makes the following diagram commute:
\begin{equation*}\begin{CD}
C_n(B^K; \mathbb{Z})@>f(K)>>M(G/K)\\
@V\bar{a}_{n}VV@VM(\widehat{a})VV \\C_n(B^H;
\mathbb{Z})@>f(H)>>M(G/H)
\end{CD}
\end{equation*}
\noindent for any $H,K \in \mathcal{H}$ with $H^a \leq K$. Here
$\bar{a}:B^K \rightarrow B^H$ is given by $\bar{a}(x)=ax$ for any
$x \in B^K$ and $\bar{a}_{\ast}$ denotes the induced map between
the chain complexes. The coboundary map
$\delta:C^n_{\mathcal{H}}(B; M) \rightarrow
C^{n+1}_{\mathcal{H}}(B; M)$ is defined by $(\delta f)(H)(\tau)=
f(H)(\partial \tau)$ for any $H\in \mathcal{H}$ and for any
$(n+1)$-cell $\tau$ of $B^H$.

\begin{defin} The Bredon cohomology
$H_{G, \mathcal{H}}^{\ast}(B; M)$ of $G$-CW-complex $B$ with
coefficients in $M$ is defined as the cohomology of the cochain
complex $C^{\ast}_{\mathcal{H}}(B; M)$.
\end{defin}

Let $\mathcal{F}=\{F_H\}_{H \in \mathcal{H}}$ be a compatible
family. Since $\mathcal{F}$ is compatible, we can consider the
universal $K$-fibration $u_K:E_K \rightarrow
B\mathrm{Aut}^I_K(F_K)$ with trivial $K$-action on the base as an
$H$-fibration with the fiber $F_H$ via conjugation action whenever
$H^a \leq K$. Let $\widetilde{a}: B\mathrm{Aut}^I_K(F_K)
\rightarrow B\mathrm{Aut}^I_H(F_H)$ be the classifying map of this
fibration. Define a contravariant functor $\underline{\pi}_{n,
\mathcal{F}}: \mathrm{Or}_{\mathcal{H}}(G) \rightarrow
\mathcal{A}b$ by letting\\ \label{local coef.}

\hspace{2cm} $\underline{\pi}_{n, \mathcal{F}}(G/H)=
\pi_n(B\mathrm{Aut}^I_H(F_H))$

\hspace{2cm}$ \underline{\pi}_{n, \mathcal{F}}(\hat{a}:G/H
\rightarrow G/K)= \pi_{\ast}(\widetilde{a}):
\pi_n(B\mathrm{Aut}^I_K(F_K)) \rightarrow
\pi_n(B\mathrm{Aut}^I_H(F_H))$.\\

From now on we assume that $B^H$ is simply connected for every $H
\in \mathrm{Iso}(B)$, where $\mathrm{Iso}(B)$ is the set of
isotropy subgroups of $B$. Let $p:E \rightarrow B_n$ be a
$G$-fibration over the $n$-skeleton of $B$ with the fiber type
$\mathcal{F}=\{F_H\}_{H\in \mathcal{H}}$ where $n \geq 2$. For
every $H \in \mathcal{H}$, the map $p_{H}: p^{-1}(B_n^H)
\rightarrow B^H_n$ is an $H$-fibration. In particular, the
$H$-fibration $p_H$ is classified by a map $\phi_{p, H}:B_n^H
\rightarrow B \mathrm{Aut}^I_HF_H$ when $H$ is an isotropy
subgroup. For an arbitrary $H \in \mathcal{H}$ and $(n+1)$-cell
$\sigma$ of $B^H_n$ with an attaching map $f_{\sigma}:
\mathbb{S}^{n} \rightarrow B_n^{G_{\sigma}} \subseteq B_n^H$, the
$G_{\sigma}$-fibration $f_{\sigma}^{\ast}(p_H)$ is classified by
the map $\mathrm{Res}_H^{G_{\sigma}} \circ \phi_{p,G_{\sigma}}
\circ f_{\sigma}$. Here, $\mathrm{Res}_H^{G_{\sigma}}: B
\mathrm{Aut}_{G_{\sigma}}^IF_{G_{\sigma}} \rightarrow B
\mathrm{Aut}^I_HF_H$ is induced by the relation $H \leq
G_{\sigma}$. When $H \in \mathrm{Iso(B)}$, the maps
$\mathrm{Res}_{H}^{G_{\sigma}}\circ \phi_{p,G_{\sigma}} \circ
f_{\sigma}$ and $\phi_{p, H} \circ f_{\sigma}$ are homotopic since
they both classify the $H$-fibration $f_{\sigma}^{\ast}(p_H)$.
Moreover, when $H,K \in \mathcal{H}$ with $H^a \leq K$, we have
$\widetilde{a}\circ \mathrm{Res}^{G_{\sigma}}_K \simeq
\mathrm{Res}^{G_{a\sigma}}_H \circ \widetilde{a}$ for every
$(n+1)$-cell $\sigma$ of $B^K$.

We define
$$\alpha_{p} \in \prod_{H \in \mathcal{H}}\mathrm{Hom}_{\mathbb{Z}}
(C_{n+1}(B^H), \pi_n(B \mathrm{Aut}^I_HF_H)) $$ by
$\alpha_{p}(H)(\sigma)=(\mathrm{Res}_{H}^{G_{\sigma}})_{\ast}[\phi_{p,G_{\sigma}}
\circ f_{\sigma}]$ for every $H \in \mathcal{H}$ and for every
$(n+1)$-cell $\sigma$ of $B^H$ with an attaching map
$f_{\sigma}:\mathbb{S}^n \rightarrow B_n^{G_{\sigma}}$. For
$\alpha_{p}$ to be a cochain, the following diagram must commute
up to homotopy

\begin{equation*}\begin{CD}
\mathbb{S}^n@>f_{\sigma}>>B_n^{G_{\sigma}}@>\phi_{p, G_{\sigma}}>>B \mathrm{Aut}^I_{G_{\sigma}}F_{G_{\sigma}}\\
@|@V\bar{a}VV@V\widetilde{a}VV\\
\mathbb{S}^n@>f_{a\sigma}>> B_n^{G_{a \sigma}}@>\phi_{p,
G_{a\sigma}}>> B \mathrm{Aut}^I_{G_{a \sigma}}F_{G_{a \sigma}}
\end{CD}
\end{equation*} for every $a \in
G$. The first square commutes because $B$ has a $G$-CW-complex
structure. Since $\widetilde{a}$ is the classifying map of the
$G_{a\sigma}$-fibration $u_{G_{\sigma}}$ and $\phi_{p,G_{\sigma}}$
is the classifying map of $p_{G_{\sigma}}$, the pullback of the
universal $G_{a \sigma}$-fibration $u_{G_{a \sigma}}$ by the
composition $\widetilde{a} \circ \phi_{p,G_{\sigma}}$ is
$G_{a\sigma}$-fiber homotopy equivalent to the fibration
$p_{G_{\sigma}}$ considered as an $G_{a\sigma}$-fibration. On the
other hand, the pullback of the $G_{a\sigma}$-fibration
$p_{G_{a\sigma}}$ by $\bar{a}$ is $G_{a \sigma}$-fiber homotopy
equivalent to the fibration $p_{G_{\sigma}}$ and hence $(\phi_{p,
G_{a\sigma}} \circ \bar{a})^{\ast}u_{G_{a\sigma}}
\simeq_{G_{a\sigma}}p_{G_{\sigma}}$. Therefore the
$G_{a\sigma}$-fibrations $(\widetilde{a} \circ
\phi_{p,G_{\sigma}})^{\ast} u_{G_{a\sigma}}$ and $(\phi_{p, G_{ a
\sigma}} \circ \bar{a})^{\ast}u_{G_{a\sigma}}$ are
$G_{a\sigma}$-fiber homotopy equivalent. By Theorem
\ref{classification theorem}, the maps $\widetilde{a} \circ
\phi_{p, G_{\sigma}}$ and $\phi_{p, G_{a\sigma}} \circ \bar{a}$
are homotopic and hence $\alpha_{p}$ is a cochain in
$C^{n+1}_{\mathcal{H}}(B, \underline{\pi}_{n, \mathcal{F}}).$
\begin{prop}\label{prop:cocycle} $\alpha_{p}$ is a cocycle.
\end{prop}
\begin{proof} For an $H \in \mathrm{Iso}(B)$, $\alpha_{p}(H) \in C^{n+1}
(B^H,\pi_n(B\mathrm{Aut}^I_HF_H))$ is the obstruction cochain for
extending the map $\phi_{p, H}:B_n^H \rightarrow B
\mathrm{Aut}^I_HF_H$ to the $(n+1)$-skeleton  of $B^H$. Therefore,
by classical obstruction theory, we have
$$(\delta \alpha_{p})(H)(\tau)=\delta(\alpha_p(H))(\tau)=0$$
for any $(n+2)$-cell $\tau$ of $B^H$. On the other hand, for
arbitrary $H \in \mathcal{H}$, we have $(\delta
\alpha_{p})(H)(\tau)= (\mathrm{Res}^{G_{\tau}}_H)_{\ast}
\big(\delta \alpha_p (G_{\tau})(\tau)\big)=0$. So, $\delta
\alpha_{p}=0$.
\end{proof}

We call $\alpha_{p}$ the {\it obstruction cocycle}. From now on we
assume that the total spaces of $G$-fibrations that we consider
have the homotopy type of a $G$-CW-complex unless otherwise
stated.

\begin{prop}\label{prop:1partofmain} A $G$-fibration $p:E_n \rightarrow B_n$ with the fiber type $\mathcal{F}=\{F_H\}_{H \in
\mathcal{H}}$, where $F_H$ is a finite $H$-CW-complex, extends to
a $G$-fibration over the (n+1)-skeleton $B_{n+1}$ if and only if
$\alpha_{p}=0$. Moreover if $E_n$ has the $G$-homotopy type of a
finite $G$-CW-complex then the total space of the fibration that
we obtain has the $G$-homotopy type of a finite $G$-CW-complex.
\end{prop}
\begin{proof} If the obstruction cocycle is zero, then $[\phi_{G_{\sigma}}\circ f_{\sigma}]=0$ for any $(n+1)$-cell
$\sigma$ of $B$. By Theorem \ref{classification theorem}, 
the restriction $p|_{\partial \sigma}$ is $G_{\sigma}$-fiber
homotopy equivalent to the trivial $G_{\sigma}$-fibration
$\varepsilon: \partial \sigma \times F_{G_{\sigma}} \rightarrow
\partial \sigma$. Let $\beta_{\sigma}: \partial \sigma \times F_{G_{\sigma}} \rightarrow p^{-1}(\partial \sigma)$ be the
$G_{\sigma}$-fiber homotopy equivalence between them. By Corollary
\ref{maincor}, $p|_{\partial \sigma}$ extends to a
$G_{\sigma}$-fibration $\tilde{p}_{\sigma}: Z \rightarrow \sigma$
over $\sigma$. Let us define
$$E'=\Big(\coprod_{\sigma \in I_{n+1}} G
\times_{G_{\sigma}}  \big ( (Z \cup_{i_1} p^{-1}(\partial
\sigma)\times I ) \big) \Big) \cup_{i_2} E$$where $I_{n+1}$ is the
set of representatives of $G$-orbits of $(n+1)$-cells and $i_j$'s
are the corresponding inclusions for $j=1,2$. Let $q: E'
\rightarrow B_{n+1}$ be defined by relations
\begin{eqnarray*}q|_{Z}= \tilde{p}_{\sigma}, \
q|_{p^{-1}(\partial \sigma)\times I}=  p|_{\partial \sigma} \times
\mathrm{id} \ \mathrm{and} \ q|_E = p .
\end{eqnarray*}
By Theorem \ref{uniformization theorem}, the $G$-map $q$ is a
$G$-fibration.

Since all the fibers and the base space have the $G$-homotopy type
of a $G$-CW-complex, $E'$ is $G$-homotopy equivalent to a
$G$-CW-complex. More precisely, for each orbit representative
$\sigma \in I_{n+1}$, we can deform $G \times_{G_{\sigma}} \big( Z
\cup_{i_1} ( p^{-1}(\partial \sigma)\times I ) \big) \cup_{i_2} E$
to $ G \times_{G_{\sigma}} \big( \sigma \times F_{G_{\sigma}} \big
) \cup_{\beta_{\sigma}} E$ relative to $E$ via a $G$-map as shown
in Figure 1.
\begin{center}
\begin{tikzpicture}[node distance=0.7cm, auto]
 \draw [rounded corners=15pt,very thick] (-2,4.3)--(-4,4.5)--(-4.5,5.5)--(-3.8,6.6);
 \draw [rounded corners=15pt,very thick] (-2,7.8)--
 (-1,8.5)--(2,9)--(5.5,8.5)--(7,7)--(6.7,5)--(4,4.5)--(2,4.4)--(0,4)--(-2,4.3);
 \draw [rounded corners=10pt,thick] (-2.5,6)--(-3.5,6);
\draw [rounded corners=5pt,thick]
(-4,8)--(-4,7)--(-3.5,6)--(-2.5,6)--(-2,7)--(-2,8) ; \draw[thick]
(-4,7)--(-2,7); \draw[dashed] (-4,8)--(-3.5,8); \draw[dashed]
(-2.5,8)--(-2,8); \draw[thick]
(-3.5,8)--(-2.5,8)--(-2.5,8.5)--(-3.5,8.5)--cycle; \draw[thick]
(-2.5,8.5) arc (0:180:0.5cm and 0.7cm); \draw[->] (-3.6,7.8) to
(-3.6, 7.2); \draw[->] (-3.2,7.8) to (-3.2, 7.2); \draw[->]
(-2.8,7.8) to (-2.8, 7.2);\draw[->] (-2.4,7.8) to (-2.4, 7.2);
\draw[->,bend left,double] (-1.5,7.7) to (0.5,7.4); \draw[->,bend
left,double] (2,7) to (4, 6.5); \draw[rounded corners=5pt,thick]
(0,6.5)--(0.5,5.5)--(1.5,5.5)--(2,6.5); \draw[thick]
(2,6.5)--(0,6.5); \draw[thick] (1.5,6.5) arc (0:180:0.5cm and
0.8cm); \draw [rounded corners=10pt,thick] (0.5,5.5)--(1.5,5.5);
\draw[->] (0.6,6.3) to (0.6,5.7); \draw[->] (1,6.3) to (1,5.7);
\draw[->] (1.4,6.3) to (1.4,5.7);\draw[thick] (5,5.5) arc
(0:180:0.6cm and 0.8cm); \draw[thick] (5,5.5) parabola bend
(4.4,5.3) (3.8,5.5); \draw[dashed] (5,5.5) parabola bend (4.4,5.7)
(3.8,5.5); \path (1.5, 3.3) node (Z) {$\mathrm{Figure \ 1}$};
\end{tikzpicture}
\end{center}
\end{proof}

Proposition \ref{prop:1partofmain} proves the first part of
Theorem \ref{mainthmofourobstructiontheory}. The second part of
the theorem says that if $\alpha_{p}$ is cohomologous to zero,
that is, $\alpha_p= \delta d$ for some cochain $d \in
C^n_{\mathcal{H}}(B, \underline{\pi}_{n, \mathcal{F}})$ then the
$G$-fibration $p|_{B_{n-1}}:p^{-1}(B_{n-1}) \rightarrow B_{n-1}$
extends to a $G$-fibration over $B_{n+1}$. In order to prove this,
we redefine $p$ over the $n$-skeleton relative to the
$(n-1)$-skeleton in such a way that the obstruction cocycle of
this new $G$-fibration vanishes.

Let $p_1$ and $p_2$ be $G$-fibrations over $B_n$ whose
restrictions to $B_{n-1}$ are $G$-fiber homotopy equivalent. Then
by Corollary \ref{cor:genel}, there is a $G$-fibration $q :Z
\rightarrow (B \times I)_n$ with $q|_{B_n \times \{0\}}=p_1$ and
$q|_{B_n \times \{1\}}= p_2$. Let $\Psi_{q,H}: (B^H \times I)_n
\rightarrow B \mathrm{Aut}^I_HF_H$ be the classifying map of the
fibration $q_H$ for $H \in \mathrm{Iso}(B)$. Note that the
composition $\Psi_{q,H} \circ i_j$ gives a classifying map for the
$H$-fibration $p_j^H$ where $i_j$'s are the corresponding
inclusions. As before, the map $\mathrm{Res}^{G_{ae}}_H \circ
\Psi_{q,G_{ae}} \circ f_{ae}$ is homotopic to $\tilde{a}\circ
\mathrm{Res}^{G_e}_K \circ \Psi_{q,G_e}\circ f_e$ for every
$(n+1)$-cell $e$ of $B^K\times I$ with the attaching map $f_e:
\mathbb{S}^{n-1} \rightarrow B^{G_{\sigma}}_n$ and for every $H^a
\leq K$. Therefore, the map $d_{p_1, q, p_2} \in \prod_{H \in
\mathcal{H}}\mathrm{Hom}_{\mathbb{Z}}(C_n(B^H), \pi_n(B
\mathrm{Aut}^I_HF_H))$ defined for an $n$-cell $\tau$ of $B^H$ by
$$d_{p_1, q , p_2}(H)(\tau)=(-1)^{n+1}(\mathrm{Res}^{G_{\tau}}_H)_{\ast}[\Psi_{q,G_{\tau}} \circ f_{\tau \times
I}],$$ is a cochain in $C^n_{\mathcal{H}}(B; \underline{\pi}_{n,
\mathcal{F}})$. We call $d_{p_1, q, p_2}$ the {\it difference
cochain}. As in the standard theory, we have the following.

\begin{prop} $\delta d_{p_1, q, p_2}= \alpha_{p_1}-\alpha_{p_2}.$
\end{prop}

\begin{proof} It suffices to show that $d_{p_1,q,p_2}(H)=\alpha_{p_1}(H)-\alpha_{p_2}(H)$
for every $H \in \mathrm{Iso}(B)$. Let $o_{\Psi} \in C^{n+1}(B^H
\times I, \pi_n(B\mathrm{Aut}^I_HF_H))$ be the obstruction cocycle
to the extension of $\Psi_{q,H}$ to $(B^H \times I)_{n+1}$. Then
$\delta o_{\Psi}=0$. On the other hand, $o_{\Psi}(H)(\theta)=
[\Psi_{q,H} \circ f_{\theta}]$ for every $(n+1)$-cell $\theta$ of
$B^H \times I$. Therefore, as in the proof of the corresponding
result in the standard theory, we have
\begin{eqnarray*} 0&=&(\delta o_{\Psi })(H)(\sigma \times
I)\\ &=& o_{\Psi }(H)(\partial \sigma \times
I)+(-1)^{n+1}\big(o_{\Psi}(H)(\sigma
\times \{1\})-o_{\Psi}(H)(\sigma \times \{0\}) \big)\\
&=&[\Psi_{q,H} \circ f_{\partial \sigma \times I}]+(-1)^{n+1}
\big([\Psi_{q,H} \circ f_{\sigma \times \{1\}}]-[\Psi_{q,H} \circ
f_{\sigma \times
\{0\}}]\big)\\
&=&(-1)^{n+1}(d_{p_1, q, p_2}(H)(\partial \sigma)+\alpha_{p_2}(H)(\sigma)-\alpha_{p_1}(H)(\sigma))\\
&=& (-1)^{n+1}(\delta d_{p_1, q,
p_2}(H)(\sigma)+\alpha_{p_2}(H)(\sigma)-\alpha_{p_1}(H)(\sigma))
\end{eqnarray*} for any $(n+1)$-cell $\sigma$ of $B^H$. This implies that
for every $H \in G$ and for every $(n+1)-$cell $\sigma$ of $B^H$,
we have
$$\delta d_{p_1, q, p_2}(H)(\sigma)=(\alpha_{p_1}-
\alpha_{p_2})(H)(\sigma)$$ and hence $\delta d_{p_1, q,
p_2}=\alpha_{p_1}- \alpha_{p_2}$. \end{proof}

The above proposition immediately implies the ``if" direction of
the second statement in Theorem
\ref{mainthmofourobstructiontheory}. To see this, note that if the
$G$-fibration $p|_{B_{n-1}}$ extends to a $G$-fibration
$\tilde{p}$ over $B_{n+1}$ then $\delta d_{p, q,
\tilde{p}|_{B_{n}}}= \alpha_p -
\underbrace{\alpha_{\tilde{p}|_{B_n}}}_0= \alpha_p$ hence the
cohomology class $[\alpha_p]$ vanishes. For the ``only if"
direction, we need the following observation.

\begin{prop} \label{differenced}Let $p:E \rightarrow B_n$ be a $G$-fibration over
$B_n$ with the fiber type $\mathcal{F}=\{F_H\}_{H \in
\mathcal{H}}$ where $F_H$ is a finite $H$-CW-complex for every $H
\in \mathcal{H}$. Then for every $d \in C^n_{\mathcal{H}}(B;
\underline{\pi}_{n, \mathcal{F}})$, there is a $G$-fibration $q: Z
\rightarrow (B \times I)_n$ such that $d_{p, q ,\tilde{p}}=d$
where $\tilde{p}=q|_{B_n \times \{1\}}$. Moreover if $E$ has the
$G$-homotopy type of a finite $G$-CW-complex then the space
$q^{-1}(B_n \times \{1\})$ has the $G$-homotopy type of a finite
$G$-CW-complex.
\end{prop}

\begin{proof} For an $n$-cell $\tau$ of $B$,
the $G_{\tau}$-map $\bar{p}_{\tau}:p^{-1}(\bar{\tau}) \cup
p^{-1}(\partial \tau) \times I \rightarrow \bar{\tau} \cup
\partial \tau \times I$, where $\bar{p}_{\tau}|_{p^{-1}(\bar{\tau})}=p|_{p^{-1}(\bar{\tau})}$
and $\bar{p}_{\tau}|_{p^{-1}(\partial \tau) \times I}=
p|_{p^{-1}(\partial \tau)}\times \mathrm{id}$, is a
$G_{\tau}$-fibration and it is classified by the map $\phi_{p,
G_{\tau}} \circ \pi_1$ where $\pi_1: \bar{\tau} \cup
\partial \tau \times I \rightarrow \bar{\tau}$ is the projection to the first coordinate. Let $E_{\tau}$ be the
pullback of $\bar{p}_{\tau}$ by the map $$\bar{f}_{\tau} =
(f_{\tau} \circ \pi_1, f_{\tau} \times \mathrm{id}): \mathbb{D}^n
\times \{0\} \cup \mathbb{S}^{n-1} \times I \rightarrow \bar{\tau}
\cup
\partial \tau \times I$$ where $f_{\tau}:(\mathbb{D}^n, \mathbb{S}^{n-1})
\rightarrow (\bar{\tau}, \tau)$ is the characteristic map of
$\tau$.

Let $X_1=\{(x,1) \in \mathbb{D}^n \times \{1\}| \ \frac{1}{2} \leq
|x| \leq 1\}$, $X_2= \{(x,1) \in \mathbb{D}^n \times \{1\}| \
\frac{1}{4} \leq |x| \leq \frac{1}{2}\}$, and $X_3=\{(x,1) \in
\mathbb{D}^n \times \{1\}| \ 0 \leq |x| \leq \frac{1}{4}\}$. Let
$p_1: E^1_{\tau} \rightarrow X_1$ be the induced
$G_{\tau}$-fibration $(\bar{f}_{\tau} f)^{\ast}(\bar{p}_{\tau})$
where $f: X_1 \rightarrow \mathbb{D}^n \times \{0\} \cup
\mathbb{S}^{n-1} \times
I$ is given by \begin{eqnarray*} f(x,1)= \left\{%
\begin{array}{ll}
(2x- \frac{x}{|x|}, \ 4|x|-3), & \hbox{$\frac{3}{4} \leq |x| \leq 1$,} \\
(2x- \frac{x}{|x|}, \ 0), & \hbox{$ \frac{1}{2} \leq |x| \leq \frac{3}{4}$.} \\
\end{array}%
\right.
\end{eqnarray*} Note that $p_1|_{\mathbb{S}^{n-1}_1 \times \{1\}}=
f_{\tau}^{\ast}(\bar{p}|_{\partial \tau})$ and
$p_1|_{\mathbb{S}^{n-1}_{\frac{1}{2}} \times \{1\}}$ is the
trivial $G_{\tau}$-fibration with the fiber
$F=p^{-1}(f_{\tau}(0))$, where $\mathbb{S}^{n-1}_r$ is the
$(n-1)$-sphere of radius $r$.

Since $d(G_{\tau})(\tau) \in \pi_n(B
\mathrm{Aut}_{G_{\tau}}F_{G_{\tau}})$, it is represented by a map
$$\Psi_{\tau}: (\mathbb{D}^n_{\frac{1}{4}}, \mathbb{S}^{n-1}_{\frac{1}{4}}) \rightarrow (B
\mathrm{Aut}_{G_{\tau}}F_{G_{\tau}}, \ast).$$ Let $u_{G_{\tau}}'=
E \mathrm{Aut_{G_{\tau}}}F_{G_{\tau}}
\times_{\mathrm{Aut}_{G_{\tau}}F_{G_{\tau}}} F \rightarrow B
\mathrm{Aut}_{G_{\tau}}F_{G_{\tau}}$. Then the restriction of the
$G_{\tau}$-fibration $\Psi_{\tau}^{\ast}(u'_{G_{\tau}}):
E_{\tau}^2 \rightarrow X_3$ to $\mathbb{S}^{n-1}_{\frac{1}{4}}
\times \{1\}$ is the same as the trivial $G_{\tau}$-fibration with
the fiber $F$. By gluing these fibration with the trivial one over
$X_2$, we obtain a $G_{\tau}$-fibration $E_{\tau}^1 \cup F \times
X_2 \cup E^2_{\tau} \rightarrow \mathbb{D}^n \times \{1\}$ over
$\mathbb{D}^n \times \{1\}$. Let $\widetilde{p}_{\tau}:
\widetilde{E}_{\tau} \rightarrow \tau$ be the corresponding
$G_{\tau}$-fibration over $\tau$. As in the proof of Proposition
\ref{prop:1partofmain}, the following $G$-map
\begin{equation*}\begin{CD}  \widetilde{E}=\Big(\coprod_{\tau \in I_{n}} G
\times_{G_{\tau}}  \big (\widetilde{E}_{\tau} \cup_{i_1}
p^{-1}(\partial
\tau)\times I  \big) \Big) \cup_{i_2} p^{-1}(B_{n-1}) \\
@V\tilde{p}VV\\
B_n
\end{CD}
\end{equation*} is a $G$-fibration over $B_n$. Moreover, when $E$
has the $G$-homotopy type of a $G$-CW-complex, so does
$\widetilde{E}$.

Let $q: E \cup ( p^{-1}(B_{n-1})\times I) \cup \widetilde{E}
\rightarrow (B \times I)_n$ be the $G$-fibration defined by
$q|_E=p$, $q|_{\widetilde{E}}=\widetilde{p}$, and
$q|_{p^{-1}(B_{n-1})\times I}= p|_{B_{n-1}} \times \mathrm{id}$.
Then $d_{p,q,\tilde{p}}(G_{\tau})(\tau)$ is represented by the
classifying map $$\widetilde{\Psi}: \mathbb{D}^n \times \{0\}\cup
\mathbb{S}^{n-1} \times I \cup X_1 \cup X_2 \cup X_3 \rightarrow B
\mathrm{Aut}_{G_{\tau}}F_{G_{\tau}}$$ where
\begin{eqnarray*} \widetilde{\Psi}|_{\mathbb{D}^n \times \{0\} \cup \ \mathbb{S}^{n-1}\times
I}&=&\phi_{p, G_{\tau}} \pi_1 \bar{f}_{\tau}, \ \ \
\widetilde{\Psi}|_{X_1}= \phi_{p,G_{\tau}}
\pi_1  \bar{f}_{\tau}  f \\
\widetilde{\Psi}|_{X_2}&=& c_{\phi_{p, G_{\tau}}(f_{\tau}(0))}, \
\ \widetilde{\Psi}|_{X_3}=\Psi_{\tau}.
\end{eqnarray*} Here, $c_{\phi_{p, G_{\tau}}(f_{\tau}(0))}$ is the
constant map at $\phi_{p, G_{\tau}}(f_{\tau}(0))$. Since
$\widetilde{\Psi}|_{\mathbb{D}^n \times \{0\}\cup \mathbb{S}^{n-1}
\times I \cup X_1}$ is homotopic to the constant map $c_{\phi_{p,
G_{\tau}}(f_{\tau}(0))}$ relative to
$\mathbb{S}^{n-1}_{\frac{1}{2}} \times \{1\}$, the map
$\widetilde{\Psi}$ also represents $d(G_{\tau})(\tau)$. Therefore,
we have $d= d_{p,q, \tilde{p}}$.
\end{proof}

Now we can prove the ``only if" of the main theorem as follows.

\begin{proofmain} It only remains to show that if $\alpha_p$ is
cohomologous to zero then there is a $G$-fibration over $B_{n+1}$
which extends $p|_{B_{n-1}}$. Let $\alpha_p= \delta d$ for some $d
\in \mathrm{Hom}(C_n(B),\underline{\pi}_{n,\mathcal{F}})$. By
Proposition \ref{differenced}, there is a $G$-fibration $q$ over
$B\times I$ such that $d=d_{p,q,\tilde{p}}$ where
$\tilde{p}=q|_{B_{n}\times \{1\}}$. Since $\alpha_p=\delta d =
\alpha_p- \alpha_{\tilde{p}}$, we have $\alpha_{\tilde{p}}=0$ and
hence $\tilde{p}$ extends to a $G$-fibration over $B_{n+1}$.
\end{proofmain}

\begin{remark} In Theorem \ref{mainthmofourobstructiontheory}, one
can replace the assumption that $B^H$ is simply-connected for
every $H \in \mathcal{H}$ with the assumption that the map
$\pi_1(B^H) \rightarrow [F_H, F_H]_H$ is trivial for every $H \in
\mathcal{H}$. In applications, one often has fibers which are
homotopy equivalent to spheres and one can take fiber joins to
make this map trivial.
\end{remark}

\subsection*{Acknowledgements} This work is part of the author's PhD thesis at the
Bilkent University. The author is grateful to her thesis advisor
Erg\"{u}n Yal\c{c}\i n for introducing her to this problem, for
valuable discussions and for the careful reading of the first
draft. The author thanks \"Ozg\"un \"Unl\"u for his crucial
comments on this work. We also thank the referee for helpful
comments, in particular, for suggesting a simpler map which
shortens the proof of Lemma \ref{lem:mapping cylinder}.

\end{document}